\documentclass[12pt,onecoulme]{article}
\usepackage{amsfonts}
\usepackage{mathrsfs}
\usepackage{amssymb}
\usepackage{color, amsmath,amssymb, amsfonts, amstext,amsthm, latexsym}

\usepackage{amssymb, epsfig, amssymb, latexsym}
\usepackage{amsmath}
\usepackage{graphicx}
\usepackage{longtable}

\numberwithin{equation}{section}

\allowdisplaybreaks

\newtheorem{definition}{Definition}[section]
\newtheorem{theorem}{Theorem}[section]
\newtheorem{lemma}{Lemma}[section]
\newtheorem{remark}{Remark}[section]

\newtheorem{example}{Example}[section]

\title{Slow Manifolds for Multi-Time-Scale\\ Stochastic  Evolutionary Systems\thanks{This work  was
 supported by the NSFC grants (10901065, 10971225, 11028102), the Chinese
Universities' Basic Research Fund, the NSF Grant 1025422,   the open
funding of LNM at the Institute of Mechanics of Chinese Academy of
Science, the Fundamental Research Funds for the Central Universities
(HUST£ºNo.2010ZD037 and N0. 2011QNQ170), the Research Fund for the
Doctoral Program of Higher Education of China (20110142120036).}}

\author{Hongbo Fu$^1,$
 Xianming Liu$^2,$
Jinqiao Duan$^3$
 \\
1. College of Mathematics and Computer Science\\
Wuhan Textile University\\ Wuhan 430073, China \\ \emph{E-mail: hbfuhust@gmail.com}\\
  2. Department of Applied Mathematics\\ Illinois Institute of Technology \\
  Chicago, IL 60616, USA  \\\emph{E-mail: duan@iit.edu}\\
  3. School of Mathematics and Statistics\\ Huazhong University of Science and Technology\\
   Wuhan 430074, China \\ \emph{mathliuxm@yahoo.cn}}

\begin{document}

\maketitle

\pagestyle{plain}

\begin{abstract}
This article deals with invariant manifolds for infinite dimensional
random dynamical systems   with different time scales.
  Such a random system is generated by a coupled system of fast-slow
  stochastic evolutionary equations. Under suitable conditions, it is proved that
  an exponentially tracking random invariant manifold exists,
eliminating the fast motion for this coupled system. It is further
shown that if the scaling parameter tends to zero, the invariant
manifold tends to a \emph{slow manifold} which captures long time
dynamics.  As examples  the   results are applied to a few systems
of coupled parabolic-hyperbolic partial differential equations,
coupled parabolic partial differential-ordinary differential
equations, and coupled hyperbolic-hyperbolic partial differential
equations.
\medskip

\end{abstract}


\section{Introduction}
%%%%%%%%%%%%%%%%%%%%%%%%%%%%%%%%%%%%%%%%%%%%%%%%%%%%%%%%%%%%%%%%%

The theory of invariant manifolds serves as a tool for analyzing
dynamical behaviors of deterministic
  systems. It was first introduced by Hadamard
\cite{Hadamard}, then  by Lyapunov \cite{Lyapunov} and Perron
\cite{Perron} for deterministic systems.  It has been further
developed by many authors for infinite dimensional deterministic
systems; see, e.g., \cite{Ruelle, Bates, Chicone, Chow, Henry}. More
recently, invariant manifolds have been investigated for infinite
dimensional stochastic systems; see \cite{Duan1, Duan2, Lu,
Mohammed, Bensoussan, Caraballo, Chueshov, Chueshov-0} among others.

Some   systems   evolve on fast and slow time scales, and may thus
be modeled by  coupled singularly perturbed stochastic ordinary or
partial differential equations (SDEs or SPDEs).
%As far as we know, there are no publications on the dynamics of coupled %stochastic systems with multi-scales in infinite dimension, although %SDEs with fast and slow components have been studied in recent years.
%%%%%      of the form
%%%%%      \begin{equation*}%%%%%%%%%%%\label{Intro-eq}
%%%%%      \begin{cases}
%%%%%      dX_t^\epsilon=a(X_t^\epsilon,
%%%%%      Y_t^\epsilon)dt+dB_1(t),\\
%%%%%      dY_t^\epsilon=\frac{1}{\epsilon}b(X_t^\epsilon,
%%%%%      Y_t^\epsilon)dt+\frac{1}{\sqrt{\epsilon}}dB_2(t).
%%%%%      \end{cases}
%%%%%      \end{equation*}
%%%%%      Where $0<\epsilon\ll 1$ is a small parameter representing the ratio
%%%%%      of time scale in this system and $B_1(t), B_2(t)$ are mutually
%%%%%      independent Wiener processes. In this setting, $X_t^\epsilon$ is
%%%%%      referred as the ``slow'' component while $Y_t^\epsilon$ is the
%%%%%      ``fast'' component.
% Averaging methods   pioneered by Khasminskii, Freidlin and Wentzell
%  have been developed to investigate such systems.   The slow motion can be
%approximated in probability by an averaged equation in a finite time interval.
%In order to capture the essential dynamics of
%deterministic dynamical systems with different times-scale, the
%theory of invariant manifolds was introduced in \cite{Fenichel, Strygin}.
For SDEs with two time scales, Schmalfu{\ss} and Schneider \cite
{Schmalfu-Schneider} have recently investigated  random inertial
manifolds that eliminate the fast variables,  by a fixed point
technique based on a random graph transformation. They show that the
inertial manifold tends to another so-called slow manifold as the
scaling parameter goes to zero. Qualitative analysis for the
behavior of the slow manifold for slow-fast SDEs on the long time
scales can be found in Wang and Roberts \cite {wangwei}.

In the present paper,  we consider invariant manifolds for
stochastic fast-slow systems in infinite dimension. Namely we
investigate the following system of fast-slow stochastic
evolutionary equations, which could be coupled SPDEs, or   coupled
SPDEs-SDEs,
\begin{eqnarray*}\label{Fast-Equation-Stoch000}
&&\dot{x}^\epsilon=\frac{1}{\epsilon}A{x}^\epsilon+\frac{1}{\epsilon}f(x^\epsilon,
y^\epsilon)+\frac{\sigma}{\sqrt{\epsilon}}\dot{w},  \quad in \;\; H_1,\\
\label{Slow-Equation-Stoch000}
&&\dot{y}^\epsilon=B{y}^\epsilon+g(x^\epsilon, y^\epsilon),\quad in
\;\;H_2,
\end{eqnarray*}
where $A$ and $B$ are  generators of $C_0-$semigroups, the
interaction functions $f$ and $g$ are continuous. The noise process
$w=\sum\limits_{j=1}^mh_jw_j,$  where $\{w_j\}_{j=1}^m$ are
two-sided Wiener processes (or Brownian motions) taking values in
$\mathbb{R}$ and  $h_j\; (1\leq j\le m)$ are given elements in
$H_1$. The small parameter $\epsilon>0$ representing the ratio of
the two time scales. The precise conditions on these quantities will
be given in Section \ref{Framework}, which allow our framework to
deal with multiscale coupled parabolic-hyperbolic systems and
coupled hyperbolic-hyperbolic systems.

%%%%%     We note that the existence of an exponentially attracting random
%%%%%     invariant manifold for the coupled parabolic¨Chyperbolic systems is
%%%%%     obtained in  \cite{Caraballo}

It is worthy mentioning that in the situation we consider here, the
noise perturbation of the fast motion equation is  additive type.
The reason is that the problem for existence of random dynamical
systems generated by stochastic partial differential equations with
general multiplicative noise is still unsolved (for details see
\cite{Duan1}). The main goal in this paper is to establish, for
$\epsilon>0$ small enough, the existence of a random invariant
manifold $M^\epsilon$ with an exponential tracking property for the
above stochastic system. Thus as a consequence, this system can be
reduced to an evolutionary equation with a modified nonlinear term,
which is useful for describing the long time behavior of the
original coupled stochastic system. There are usually two approaches
to construction of invariant manifolds: Hadamard graph transform
method (see \cite{Schmalfuss, Duan1}) and Lyapunov-Perron method
(see \cite{Chow-2, Duan2, Caraballo}). We achieve our results by the
latter which is different from the method of random graph
transformation in \cite{Schmalfu-Schneider}. In this approach one
key assumption is that the   Lipschitz constant of the nonlinear
term in fast component is small enough comparing with the decay rate
of the linear operator $A$. In particular, under suitable conditions
it is further shown that this manifold $M^\epsilon$ can be
asymptotically approximated for $\epsilon$ sufficiently small by a
\emph{slow manifold} $M^0$ for a reduced stochastic system. We note
that, in the case of Lyapunov-Perron method applied to a coupled
stochastic systems, the existence of an random invariant manifold
for the coupled stochastic parabolic-hyperbolic equations, that do
not contain two widely separated characteristic timescales, is
obtained by Caraballo Chueshov and Langa in \cite{Caraballo}. We
remark that, whereas the existence of \emph{slow manifold} is not
studied, in their paper the just mentioned authors also verify that
this random manifold converges to its deterministic counterpart when
the intensity of noise tends to zero.

%%%%%    But since the coupled systems do not contain two widely separated
%%%%%    characteristic timescales,

This paper is organized as follows. In Section 2, some basic
concepts in random dynamical systems and random invariant manifolds
are recalled. Our framework is presented in Section 3. In Section 4,
we establish the existence of a random invariant manifold
$M^\epsilon$ possessing an exponential tracking property, and then
in Section 5 we show $M^\epsilon$ converges to a slow manifold $M^0$
with rate of order 1. Section 6 is devoted to a few illustrative
examples. Remarks on local manifolds for systems with local
Lipschitz nonlinearities are given in Section 7.

\section{Preliminaries on random dynamical systems}  \label{RDS}
%%%%%%%%%%%%%%%%%%%%%%%%%%%%%%%%%%%%%%%%%%%%%%%%%%%%%%%%%%%%%%%%%

We now recall basic concepts in random dynamical systems (RDS) and
random invariant manifolds (RIM). For more details, see
\cite{Arnold, Duan1, Duan2}.

\begin{definition}
Let $(\Omega, \mathscr{F}, \mathbb{P})$ be a probability space and
$\theta=\{\theta_t\}_{t\in \mathbb{R}}$ be a flow on $\Omega$ which
is defined as a mapping
\begin{equation*}
\theta: \mathbb{R}\times \Omega \mapsto \Omega
\end{equation*}
and satisfies

$\bullet$  $\theta_0=id_\Omega$,

$\bullet$ $\theta_s\theta_t=\theta_{s+t}$ for all $s, t \in
\mathbb{R},$

 $\bullet$  the mapping $(t,
\omega)\mapsto\theta_t\omega$ is $(\mathscr{B}(\mathbb{R}) \times
\mathscr{F}, \mathscr{F})-$measurable and
$\theta_t\mathbb{P}=\mathbb{P}$ for all $t\in\mathbb{R}.$ Then the
quadruple $(\Omega, \mathcal{F}, \mathbb{P}, \theta)$ is called a
driving dynamical system.
\end{definition}

We will work on the driving dynamical system represented by Wiener
process. To be more precise, let $\Omega=C_0(\mathbb{R},
\mathbb{R}^m$) be the continuous paths $\omega(t)$ on $\mathbb{R}$
with values $\mathbb{R}^m$ such that $\omega(0)=0$. This set is
equipped with the compact-open topology. Let $\mathcal{F}$ be the
associated Borel $\sigma-$field and $\mathbb{P}$ be the Wiener
measure. Then we identify $\omega$ with
$$\big(w_1(t), w_2(t), \cdots, w_m(t)\big)=\omega(t),\;t\in\mathbb{R}.$$
The operators $\theta_t$ forming the flow are given by the Wiener
shift:
$$\theta_t\omega(\cdot)=\omega(\cdot+t)-\omega(t), \;\omega\in \Omega, \;t\in \mathbb{R}.$$
Note that the measure $\mathbb{P}$ is invariant with respect to the
above flow and then the quadruple $(\Omega, \mathcal{F}, \mathbb{P},
\theta)$ is  a driving dynamical system.

\medskip

\begin{definition}
Let $(\mathbb{H}, d_\mathbb{H})$ be a metric space with Borel
$\sigma-$field $\mathscr{B}(\mathbb{H})$. A cocycle is a mapping:
$$\phi:\mathbb{R}^+\times \Omega\times \mathbb{H}\mapsto \mathbb{H},$$
which is
$\left(\mathscr{B}(\mathbb{R}^+)\times\mathscr{F}\times\mathscr{B}(\mathbb{H}),
\mathscr{B}(\mathbb{H})
 \right)-$measurable such that
\begin{eqnarray*}
&&\phi(0, \omega, x)=x,\\
&&\phi(t+s, \omega, x)=\phi(t, \theta_s\omega, \phi(s, \omega, x)),
\end{eqnarray*}
for $t, s \in \mathbb{R}^+, \omega\in \Omega$, and $x\in
\mathbb{H}$. Then $\phi$ together with the driving   system $\theta$
forms a random dynamical system (RDS).
\end{definition}

A RDS is called continuous (differentiable) if $x\rightarrow\phi(t,
\omega, x)$ is continuous (differentiable) for $t\geq 0$ and $
\omega\in\Omega.$ A family of nonempty closed sets $M=\{M(\omega)\}$
contained in a metric space $(\mathbb{H}, \|\cdot\|_\mathbb{H})$ is
called a \emph{random set} if for every $y\in \mathbb{H}$ the
mapping
$$\omega\rightarrow \inf\limits_{x\in M(\omega)}\|x-y\|_\mathbb{H}$$
is a random variable. Now we introduce the random invariant manifold
concept.

\medskip

\begin{definition}
A random set $M(\omega)$ is called a positively invariant set if
$$\phi(t, \omega, M(\omega))\subset M(\theta_t\omega), \;for\; t\geq 0, \; \omega\in \Omega.$$
If $M$ can be represented as a graph of a Lipschitz mapping
$$\psi(\cdot, \omega): {H}_1\rightarrow {H}_2, \;\; \mathbb{H}={H}_1\times {H}_2$$
such that
$$M(\omega)=\{(x_1, \psi(x_1, \omega)): x_1\in {H}_1\},$$
then $M(\omega)$ is called a Lipschitz random invariant manifold.
If, in addition, for every $x\in \mathbb{H}$, there exists an $x'\in
M(\omega)$ such that for all $\omega\in \Omega$,
$$\|\phi(t, \omega, x)-\phi(t, \omega, x')\|_\mathbb{H}\leq c_1(x, x', \omega)e^{-c_2t}\|x-x'\|_\mathbb{H},\; t\geq 0,$$
 where
  $c_1$ is a positive random variable depending on $x$ and $x'$, while $c_2$ is a positive constant,
then $M(\omega)$ is said to have an exponential tracking property.
\end{definition}

%%%%%%%%%%%%%%%%%%%%%%%%%%%%%%%%%%%%%%%%%%%%%%%%%%%%%%%%
\section{Framework}\label{Framework}

Consider the following system of stochastic evolutionary equations
with two time scales
\begin{eqnarray}\label{Fast-Equation-Stoch}
&&\dot{x}^\epsilon=\frac{1}{\epsilon}A{x}^\epsilon+\frac{1}{\epsilon}f(x^\epsilon,
y^\epsilon)+\frac{\sigma}{\sqrt{\epsilon}}\dot{w},  \quad in \; H_1,\\
\label{Slow-Equation-Stoch}
&&\dot{y}^\epsilon=B{y}^\epsilon+g(x^\epsilon, y^\epsilon),\quad in
\;H_2,
\end{eqnarray}
where $A$ is a generator of a $C_0$-semigroups on separable Hilbert
space $H_1$, and $B$ is a generator of a $C_0$-groups on separable
Hilbert $H_2$. Nonlinearities $f$ and $g$ are continuous functions,
\begin{equation*}
f: H_1\times H_2\mapsto H_1, \quad g: H_1\times H_2\mapsto H_2,
\end{equation*}
with $f(0,0)=g(0,0)=0$.  The noise process
$w=\sum\limits_{j=1}^mh_jw_j$, where $\{w_j\}_{j=1}^m$ are two-sided
Wiener processes taking values in $\mathbb{R}$ and $h_j\; (1\leq
j\le m)$ are given elements in $H_1$. Moreover, $\sigma$ is a
nonzero constant (noise intensity),  and $\epsilon$ is a small
positive parameter representing the ratio of time scales in this
fast-slow system. In this setting, $x^\epsilon$ is referred as the
``fast'' component while $y^\epsilon$ is the ``slow'' component.

Denote by $\|\cdot\|_1$ and $\|\cdot\|_2$ the norms in $H_1$ and
$H_2$, respectively. The norm in $\mathbb{H}=H_1 \times H_2$ is
denoted as $\|\cdot \|$. For the linear operators $A$ and $B$ we
assume the following conditions.

\medskip

(\textbf{A1}) Let $A$ be the generator of a $C_0-$semigroup $e^{At}$
on $H_1$   satisfying
\begin{equation*}%\label{Semigroup-A}
\|e^{At}x\|_1\leq  e^{-\gamma_1t}\|x\|_1, \quad t\geq 0.
\end{equation*}
for all  $x\in H_1$, with a   constant (i.e., decay rate)
$\gamma_1>0$. Moreover, $B$ is the generator of a $C_0-$group
$e^{Bt}$ on $H_2$ satisfying
\begin{equation*}%\label{Semigroup-B}
\|e^{Bt}y\|_2\leq  e^{-\gamma_2t}\|y\|_2, \quad t\leq 0.
\end{equation*}
for all $y\in H_2$, with a constant $\gamma_2\geq 0$.

%%%%%%%%%%%  Note  %%%%%%%%%%%%%%%%%%%%%%%%%%%%%%
%\begin{remark}
%Recently, Duan and Wang \emph{\cite{Duan&Wang}} extended directly the results to infinite
%dimension. In their work, unlike in the present paper, they have
%made an assumption on operator $B$ with $\gamma_2\leq 0$. As a result, their framework can not deal with coupled %parabolic-hyperbolic systems or coupled hyperbolic-hyperbolic systems,  but only coupled parabolic partial %differential-ordinary differential systems or coupled  ordinary differential systems. In the present paper we %assume that $\gamma_2\geq 0$ and can thus handle cases of coupled parabolic-hyperbolic systems and
%coupled hyperbolic-hyperbolic systems.
%\end{remark}
%%%%%%%%%%%%%%%%%%%%%%%%%%%%%%%%

We also make the following two more assumptions.

(\textbf{A2}) Lipschitz condition: There exists a positive constant
$K$ such that for all $(x_i, y_i) \in H_1 \times H_2$
\begin{equation*}%\label{Fast-Lip}
\|f(x_1, y_1)-f(x_2, y_2)\|_1\leq K(\|x_1-x_2\|_1+\|y_1-y_2\|_2),
\end{equation*}
and
\begin{equation*}%\label{Slow-Lip}
\|g(x_1, y_1)-g(x_2, y_2)\|_2\leq K(\|x_1-x_2\|_1+\|y_1-y_2\|_2).
\end{equation*}

%For simplicity, we suppose $f(0, 0)=0$, $g(0, 0)=0$.

(\textbf{A3}) Assume that the Lipschitz constant $K$ of the
nonlinear terms in system
\eqref{Fast-Equation-Stoch}--\eqref{Slow-Equation-Stoch} is smaller
than the decay rate $\gamma_1$ of $A$, that is,
\begin{equation*}%\label{Decay-Lip}
K< \gamma_1.
\end{equation*}

\begin{remark}
We note that the system
\eqref{Fast-Equation-Stoch}--\eqref{Slow-Equation-Stoch} is an
abstract model for various complex systems under random influences,
which can be a finite-dimensional, stochastic slow-fast system
analysed in \cite{Schmalfu-Schneider, wangwei}.
%Thus, our abstract framework developed
%in the paper is quite general to partly cover the results in
%\cite{wangwei}.

%%%%%   In the particular case, where $A: \mathbb{R}^n\rightarrow
%%%%%   \mathbb{R}^n$, $B: \mathbb{R}^m\rightarrow \mathbb{R}^n$, and
%%%%%   functions $f: \mathbb{R}^n\times\mathbb{R}^m\rightarrow\mathbb{R}^n,
%%%%%   g: \mathbb{R}^n\times\mathbb{R}^m\rightarrow\mathbb{R}^m$
\end{remark}

Now as in \cite{Duan1}, we verify that the stochastic evolutionary
system \eqref{Fast-Equation-Stoch}--\eqref{Slow-Equation-Stoch} can
be transformed into a random evolutionary system which generates a
RDS. For this purpose, let $\eta^{\frac{1}{\epsilon}}$ be a
stationary solution of the linear stochastic evolutionary equation
\begin{equation}\label{Linear-equation-scale}
d\eta^\frac{1}{\epsilon}(t)=\frac{1}{\epsilon}A
\eta^\frac{1}{\epsilon}(t)dt+\frac{\sigma}{\sqrt{\epsilon}}dw(t).
\end{equation}
This means that   the random variable $\eta^{\frac{1}{\epsilon}}$
with values in $H_1$ is defined on a $\{\theta_t\}_{t\in
\mathbb{R}}-$invariant set of full measure such that
$$t  \rightarrow  \eta^{\frac{1}{\epsilon}}(\theta_t\omega)$$
is a solution version for \eqref{Linear-equation-scale}. Let $\xi$
be the stationary solution of the linear stochastic evolutionary
equation
\begin{equation*}
d\xi(t)=A\xi(t)+\sigma dw(t).
\end{equation*}
Then by the scale property of Wiener process,
$\eta^\frac{1}{\epsilon}(\theta_t\omega)$ has the same distribution
of $\xi(\theta_{\frac{t}{\epsilon}}\omega),$ see the Lemma 3.2 in
\cite{Schmalfu-Schneider}. For the existence of stationary solutions
to stochastic evolutionary equations see
  \cite{Caraballo0}.

Define $ X^\epsilon =x^\epsilon-
\eta^{\frac{1}{\epsilon}}(\theta_t\omega) $ and
 $ Y^\epsilon =y^\epsilon $. Then the original evolutionary system
\eqref{Fast-Equation-Stoch}--\eqref{Slow-Equation-Stoch} is
converted to the following random evolutionary system
\begin{eqnarray}\label{Fast-Equation-Random}
&&dX^\epsilon=\frac{1}{\epsilon}A X^\epsilon
dt+\frac{1}{\epsilon}F(X^\epsilon, Y^\epsilon,
\theta_t^\epsilon\omega)dt,\\
\label{Slow-Equation-Random}&&dY^\epsilon=BY^\epsilon
dt+G(X^\epsilon, Y^\epsilon, \theta_t^\epsilon\omega)dt,
\end{eqnarray}
where
\begin{eqnarray*}
&&F(X^\epsilon, Y^\epsilon,
\theta_t^\epsilon\omega)=f(X^\epsilon+\eta^\frac{1}{\epsilon}(\theta_t\omega),
Y^\epsilon),\\
&&G(X^\epsilon, Y^\epsilon,
\theta_t^\epsilon\omega)=g(X^\epsilon+\eta^\frac{1}{\epsilon}(\theta_t\omega),
Y^\epsilon).
\end{eqnarray*}

Let $Z^\epsilon(t, \omega, Z_0)=\big(X^\epsilon(t, \omega, X_0,
Y_0), Y^\epsilon(t, \omega, X_0, Y_0)\big)$ be the solution of
\eqref{Fast-Equation-Random}--\eqref{Slow-Equation-Random} with
initial data $\big(X^\epsilon(0), Y^\epsilon(0)\big)=(X_0,
Y_0):=Z_0.$  Then the solution operator of
\eqref{Fast-Equation-Random}--\eqref{Slow-Equation-Random}
 $$\Phi^\epsilon\big(t, \omega, (X_0, Y_0)\big)=\big(X^\epsilon(t,
\omega, X_0, Y_0), Y^\epsilon(t, \omega, X_0, Y_0)\big)$$ defines a
random dynamical system \cite{Duan1}.   Furthermore
$$\phi^\epsilon(t, \omega):=\Phi^\epsilon(t, \omega)+(\eta^\frac{1}{\epsilon}(\theta_t\omega), 0),\; t\geq 0, \;\; \omega\in \Omega$$
is the random dynamical system generated by the original system
\eqref{Fast-Equation-Stoch}--\eqref{Slow-Equation-Stoch}.

 We introduce some notations. Let $\mu$ be a positive number satisfying
\begin{equation}\label{Condition-mu}
\gamma_1-\mu>K.
\end{equation}
For any $\alpha\in \mathbb{R}$, define  Banach spaces
$$C_{\alpha}^{i,-}=\left\{\varphi:(-\infty,
0]\mapsto H_i \;\; is\; \;\;continuous \;\; and \;\;
\sup\limits_{t\leq 0}\|e^{-\alpha t}\varphi(t)\|_i< \infty\right\}$$
with the norm $\|\varphi\|_{C_{\alpha}^{i,-}}=\sup\limits_{t\leq
0}\|e^{-\alpha t}\varphi(t)\|_i$ for $i= 1, 2.$ Similarly, we define
Banach spaces
$$C_{\alpha}^{i,+}=\left\{\varphi:[0, \infty,
)\mapsto H_i \;\; is \;\;\; continuous \;\; and \;\;
\sup\limits_{t\geq 0}\|e^{-\alpha t}\varphi(t)\|_i< \infty\right\}$$
with the norm $\|\varphi\|_{C_{\alpha}^{i,+}}=\sup\limits_{t\geq
0}\|e^{-\alpha t}\varphi(t)\|_i$ for $i= 1, 2.$ Let
$C_{\alpha}^{\pm}$ be the product Banach spaces
$C_{\alpha}^{\pm}:=C_{\alpha}^{1,\pm}\times C_{\alpha}^{2,\pm} $,
with the norm
$$\|z\|_{C_{\alpha}^{\pm}}=\|x\|_{C_{\alpha}^{1,\pm}}+\|y\|_{C_{\alpha}^{2,\pm}}, \; z=(x,y)\in C_{\alpha}^{\pm}.$$

%%%%%%%%%%%%%%%%%%%%%%%%%%%%%%%%%%%%%%%%%%%%%%%%%%%%%%%%%%%%%%%
\section{Slow manifolds}%Construction of random invariant manifold

In this section, we   prove the existence of a Lipschitz continuous
invariant manifolds $M^\epsilon(\omega)$, with
  an \emph{exponential tracking property},  for the random
evolutionary system
\eqref{Fast-Equation-Random}--\eqref{Slow-Equation-Random}.

Define
$$M^\epsilon(\omega) \triangleq \left\{Z_0\in
\mathbb{H}: \; Z^\epsilon(\cdot,\omega, Z_0)\in
C_{-\frac{\mu}{\epsilon}}^{-}\right\}.
$$
This is the set of all initial data through which solutions are
bounded by $e^{-\frac{\mu}{\epsilon}t}$. We shall use
Lyapunov-Perron method to prove that $M^\epsilon(\omega)$ is an
invariant manifold described by the graph of a Lipschitz function.
For this we will need the following properties of the random
function $Z^\epsilon(\cdot,\omega, Z_0)$ (see \cite{Duan2}).

\begin{lemma}\label{If and only if}
Suppose that $Z^\epsilon(\cdot, \omega)=\big(X^\epsilon(\cdot,
\omega), Y^\epsilon(\cdot, \omega)\big)$ is in
$C_{-\frac{\mu}{\epsilon}}^{-}$. Then $Z^\epsilon(t, \omega)$ is the
solutions of
\eqref{Fast-Equation-Random}--\eqref{Slow-Equation-Random} with
initial data $Z_0=(X_0, Y_0)$ if and only if $Z^\epsilon(\cdot,
\omega)$ satisfies
\begin{equation*}
\left(
\begin{array}{ccc}
X^\epsilon(t)\\
\\
Y^\epsilon(t)
\end{array}
\right)= \left(
\begin{array}{ccc}
\frac{1}{\epsilon}\int_{-\infty}^te^{\frac{A(t-s)}{\epsilon}}F(X^\epsilon(s), Y^\epsilon(s), \theta_s^\epsilon\omega)ds\\
\\
e^{Bt}Y_0+\int_0^te^{B(t-s)}G(X^\epsilon(s), Y^\epsilon(s),
\theta_s^\epsilon\omega)ds
\end{array}
\right).
\end{equation*}
\end{lemma}

\begin{theorem}\label{Th-RIM} (Slow manifolds)\\
Assume that (\textbf{A1})--(\textbf{A3}) hold and   that
$\epsilon>0$ is sufficiently small. Then the random dynamical system
defined by
\eqref{Fast-Equation-Random}--\eqref{Slow-Equation-Random} has a
Lipschitz random slow manifold $M^\epsilon(\omega)$ represented as a
graph
$$
M^\epsilon(\omega)=\left\{\big(H^\epsilon(\omega, Y_0), Y_0\big):
Y_0\in H_2\right\},
$$
where
$$
H^\epsilon(\cdot, \cdot): \Omega\times H_2\mapsto H_1
$$
is the graph mapping with Lipschitz constant satisfying
$$
Lip H^\epsilon(\omega,
\cdot)\leq\frac{K}{\left(\gamma_1-\mu\right)\left[1-K\left(\frac{1}{\gamma_1-\mu}
+\frac{\epsilon}{\mu-\epsilon\gamma_2}\right)\right]}, \;\omega\in
\Omega.
$$
\end{theorem}

\begin{proof}
We adapt the method of Lyapunov-Perron   to fast-slow random
dynamical systems. To construct an invariant manifold for system
\eqref{Fast-Equation-Random}--\eqref{Slow-Equation-Random} we first
consider integral equations
\begin{equation}\label{Equi-Fast-Slow-Random}
\left(
\begin{array}{ccc}
X^\epsilon(t)\\
\\
Y^\epsilon(t)
\end{array}
\right)= \left(
\begin{array}{ccc}
\frac{1}{\epsilon}\int_{-\infty}^te^{\frac{A(t-s)}{\epsilon}}F(X^\epsilon(s), Y^\epsilon(s), \theta_s^\epsilon\omega)ds\\
\\
e^{Bt}Y_0+\int_0^te^{B(t-s)}G(X^\epsilon(s), Y^\epsilon(s),
\theta_s^\epsilon\omega)ds
\end{array}
\right), \;\; t\leq 0.
\end{equation}
A solution of  \eqref{Equi-Fast-Slow-Random} is denoted by
$Z^\epsilon(t, \omega, Z_0)=\big(X^\epsilon(t, \omega, Y_0),
Y^\epsilon(t, \omega, Y_0)\big)$. Introduce the operators $\mathcal
{J}_1^\epsilon: C_{-\frac{\mu}{\epsilon}}^{-}\mapsto
C_{-\frac{\mu}{\epsilon}}^{1,-}$ and $\mathcal {J}_2^\epsilon:
C_{-\frac{\mu}{\epsilon}}^{-}\mapsto
C_{-\frac{\mu}{\epsilon}}^{2,-}$ by means of

\begin{equation*}%%%%%     \label{Operator-J_1}
\mathcal
{J}_1^\epsilon(z(\cdot))[t]=\frac{1}{\epsilon}\int_{-\infty}^te^{\frac{A(t-s)}{\epsilon}}F(x(s),
y(s), \theta_s^\epsilon\omega)ds,
\end{equation*}

\begin{equation*}%%%%%     \label{Operator-J_2}
\mathcal
{J}_2^\epsilon(z(\cdot))[t]=e^{Bt}Y_0+\int_0^te^{B(t-s)}G(x(s),
y(s), \theta_s^\epsilon\omega)ds,
\end{equation*}
for $t\leq 0$ and define the mapping $\mathcal{J}^\epsilon$ by
\begin{equation*}%%%%%     \label{Operator-J}
\mathcal{J}^\epsilon(z(\cdot)):= \left(
\begin{array}{ccc}
\mathcal{J}_1^\epsilon(z(\cdot))\\
\\
\mathcal{J}_2^\epsilon(z(\cdot))
\end{array}
\right).
\end{equation*}
It can be verified that $\mathcal{J}^\epsilon$ maps
$C_{-\frac{\mu}{\epsilon}}^{-}$ into itself. To this end, taking
$z=(x, y)\in C_{-\frac{\mu}{\epsilon}}^{-}$, we have that
\begin{eqnarray}
\nonumber
\|\mathcal{J}_1^\epsilon(z)\|_{{C_{-\frac{\mu}{\epsilon}}^{1,-}}}
&\leq&\frac{K}{\epsilon}\sup\limits_{t\leq
0}\bigg\{e^{\frac{\mu}{\epsilon}t}\int_{-\infty}^te^{\frac{-\gamma_1(t-s)}{\epsilon}}\big(\|x(s)\|_1
+\|y(s)\|_2\big)ds\bigg\}\\
\nonumber &\leq&\frac{K}{\epsilon}\sup\limits_{t\leq 0}\bigg\{\int
_{-\infty}^te^{(\frac{-\gamma_1}{\epsilon}+\frac{\mu}{\epsilon})(t-s)}ds\bigg\}\|z\|_{{C_{-\frac{\mu}{\epsilon}}^{-}}}\\
&=&\frac{K}{\gamma_1-\mu}\|z\|_{C^-_{-\frac{\mu}{\epsilon}}},
\end{eqnarray}
and
\begin{eqnarray}
\nonumber
\|\mathcal{J}_2^\epsilon(z)\|_{{C_{-\frac{\mu}{\epsilon}}^{2, -}}}
&\leq&{K}\sup\limits_{t\leq 0}\bigg\{e^{\frac{\mu}{\epsilon}
t}\int_{t}^0e^{{-\gamma_2(t-s)}}e^{-\frac{\mu}{\epsilon}s}ds
\bigg\}\|z\|_{{C_{-\frac{\mu}{\epsilon}}^{-}}}+\sup\limits_{t\leq
0}\left\{e^{\frac{\mu}{\epsilon}
t}\cdot e^{-\gamma_2t}\|Y_0\|_2\right\}\\
\nonumber &\leq&{K}\sup\limits_{t\leq 0}\bigg\{\int
_{t}^0e^{(-\gamma_2+\frac{\mu}{\epsilon})(t-s)}ds\bigg\}\|z\|_{{C_{-\frac{\mu}{\epsilon}}^{-}}}+\|Y_0\|_2\\
&=&\frac{\epsilon
K}{\mu-\epsilon\gamma_2}\|z\|_{{C_{-\frac{\mu}{\epsilon}}^{-}}}+\|Y_0\|_2.
\end{eqnarray}
Hence,   by definition of  $\mathcal{J}^\epsilon$ we obtain
\begin{equation*}
\|\mathcal{J}^\epsilon(z)\|_{{C_{-\frac{\mu}{\epsilon}}^{-}}}\leq\kappa(K,
\gamma_1, \gamma_2, \mu,
\epsilon)\|z\|_{C^-_{-\frac{\mu}{\epsilon}}}+\|Y_0\|_2
\end{equation*}
with
\begin{eqnarray*}
\nonumber\kappa(K, \gamma_1, \gamma_2, \mu,
\epsilon)&=&\frac{K}{\gamma_1-\mu}+\frac{\epsilon
K}{\mu-\epsilon\gamma_2}.
\end{eqnarray*}
Thus,   we   conclude that $\mathcal{J}^\epsilon$ maps
$C_{-\frac{\mu}{\epsilon}}^{-}$ into itself.

 Next we show that
the mapping $\mathcal{J}^\epsilon$ is contractive. To this end,
taking $z=(x, y), \bar{z}=(\bar{x}, \bar{y}) \in
C_{-\frac{\mu}{\epsilon}}^{-}$, we have that
\begin{eqnarray}
\nonumber
\|\mathcal{J}_1^\epsilon(z)-\mathcal{J}_1^\epsilon(\bar{z})\|_{{C_{-\frac{\mu}{\epsilon}}^{1,-}}}
&\leq&\frac{K}{\epsilon}\sup\limits_{t\leq
0}\bigg\{e^{\frac{\mu}{\epsilon}t}\int_{-\infty}^te^{\frac{-\gamma_1(t-s)}{\epsilon}}\big(\|x(s)-\bar{x}(s)\|_1\\
\nonumber&&+\|y(s)-\bar{y}(s)\|_2\big)ds\bigg\}\\
\nonumber &\leq&\frac{K}{\epsilon}\sup\limits_{t\leq 0}\bigg\{\int
_{-\infty}^te^{(\frac{-\gamma_1}{\epsilon}+\frac{\mu}{\epsilon})(t-s)}ds\bigg\}\|z-\bar{z}\|_{{C_{-\frac{\mu}{\epsilon}}^{-}}}\\
&=&\frac{K}{\gamma_1-\mu}\|z-\bar{z}\|_{C_{-\frac{\mu}{\epsilon}}},\label{Operator-J_1-Contractive}
\end{eqnarray}
and
\begin{eqnarray}
\nonumber
\|\mathcal{J}_2^\epsilon(z)-\mathcal{J}_2^\epsilon(\bar{z})\|_{{C_{-\frac{\mu}{\epsilon}}^{2,
-}}} &\leq&{K}\sup\limits_{t\leq
0}\bigg\{e^{\frac{\mu}{\epsilon} t}\int_{t}^0e^{{-\gamma_2(t-s)}}e^{-\frac{\mu}{\epsilon}s}ds \bigg\}\|z-\bar{z}\|_{{C_{-\frac{\mu}{\epsilon}}^{-}}}\\
\nonumber &\leq&{K}\sup\limits_{t\leq 0}\bigg\{\int
_{t}^0e^{(-\gamma_2+\frac{\mu}{\epsilon})(t-s)}ds\bigg\}\|z-\bar{z}\|_{{C_{-\frac{\mu}{\epsilon}}^{-}}}\\
&=&\frac{\epsilon
K}{\mu-\epsilon\gamma_2}\|z-\bar{z}\|_{{C_{-\frac{\mu}{\epsilon}}^{-}}}.\label{Operator-J_2-Contractive}
\end{eqnarray}
Hence,   by \eqref{Operator-J_1-Contractive} and
\eqref{Operator-J_2-Contractive}
\begin{equation*}
\|\mathcal{J}^\epsilon(z)-\mathcal{J}^\epsilon(\bar{z})\|_{{C_{-\frac{\mu}{\epsilon}}^{-}}}\leq\kappa(K,
\gamma_1, \gamma_2, \mu,
\epsilon)\|z-\bar{z}\|_{C^-_{-\frac{\mu}{\epsilon}}},
\end{equation*}
where
\begin{eqnarray*}
\nonumber\kappa(K, \gamma_1, \gamma_2, \mu,
\epsilon)&=&\frac{K}{\gamma_1-\mu}+\frac{\epsilon
K}{\mu-\epsilon\gamma_2} \rightarrow \frac{K}{\gamma_1-\mu}
\end{eqnarray*}
as $\epsilon\rightarrow0$. Taking into account of
\eqref{Condition-mu} there is a sufficiently small constant
$\epsilon_0>0$ such that
$$\kappa(K,
\gamma_1, \gamma_2, \mu, \epsilon)< 1, \;for\; \epsilon\in (0,
\epsilon_0].$$ Therefore, the mapping $\mathcal {J}^\epsilon$ is
strictly contractive in ${{C_{-\frac{\mu}{\epsilon}}^{-}}}$ , and,
consequently, the integral equation \eqref{Equi-Fast-Slow-Random}
has a unique solution $Z^\epsilon(t, \omega, Y_0)=\big(X^\epsilon(t,
\omega, Y_0), Y^\epsilon(t, \omega, Y_0)\big)$ in
${{C_{-\frac{\mu}{\epsilon}}^{-}}}$. Furthermore one has the
estimate
\begin{equation}\label{Fixedpoint-Lip}
\|Z^\epsilon(\cdot, \omega, Y_1)-Z^\epsilon(\cdot, \omega,
Y_2)\|_{{C_{-\frac{\mu}{\epsilon}}^{-}}}\leq \frac{1}{1-\kappa(K,
\gamma_1, \gamma_2, \mu, \epsilon)}\|Y_1-Y_2\|_2
\end{equation}
for all $\omega\in \Omega, Y_1, Y_2 \in H_2.$\\

Define
\begin{equation}\label{H-epsilon}
H^\epsilon(\omega,
Y_0)=\frac{1}{\epsilon}\int_{-\infty}^0e^{-As/\epsilon}F\big(X^\epsilon(s,
\omega, Y_0), Y^\epsilon(s, \omega, Y_0),
\theta_s^\epsilon\omega\big)ds,
\end{equation}
we then get from \eqref{Fixedpoint-Lip}
\begin{equation*}
\|H^\epsilon(\omega, Y_1)-H^\epsilon(\omega,
Y_2)\|_1\leq\frac{K}{\big(\gamma_1-\mu\big)}\frac{1}{\left[1-\kappa(K,
\gamma_1, \gamma_2, \mu, \epsilon)\right]}\|Y_1-Y_2\|_2
\end{equation*}
for all $Y_1, Y_2 \in H_2, \omega\in \Omega.$ It then follows from
Lemma \ref{If and only if} that
$$M^\epsilon(\omega)=\left\{\big(H^\epsilon(\omega, Y_0), Y_0\big):
Y_0\in H_2\right\}.$$ In order to see that $M^\epsilon(\omega)$ is a
random set we need to show that for any $z=(x, y)\in
\mathbb{H}=H_1\times H_2$,
\begin{equation}\label{inf}
\omega\rightarrow \inf\limits_{z'\in \mathbb{H}}\|(x,
y)-(H^\epsilon(\omega, \mathcal{P}z'), \mathcal{P}z')\|
\end{equation}
is measurable, see Castaing and Valadier \cite{Castaing}, Theorem
III.9. Let $\mathbb{H}_c$ be a countable dense set of the separable
space $\mathbb{H}$. Then the right hand side of \eqref{inf} is equal
to
\begin{equation}
\inf\limits_{z'\in \mathbb{H}_c}\|(x, y)-(H^\epsilon(\omega,
\mathcal{P}z'), \mathcal{P}z')\|
\end{equation}
which follows immediately by the continuity of $H^\epsilon(\omega,
\cdot)$. The measurability of any expression under the infimum of
\eqref{inf} follows since $\omega\rightarrow H^\epsilon(\omega,
\mathcal{P}z')$ is measurable for any $z'\in \mathbb{H}.$

It remains to show that $M^\epsilon(\omega)$ is invariant, i.e., for
each $Z_0=(X_0, Y_0)\in M^\epsilon(\omega)$, $Z^\epsilon(s, \omega,
Z_0)\in M^\epsilon(\theta_s^\epsilon\omega)$ for all $s\geq 0$. We
first note that for each fixed $s\geq 0$, $Z^\epsilon(t+s, \omega,
Z_0)$ is a solution of
\begin{eqnarray*}
&&dX^\epsilon=\frac{1}{\epsilon}A X^\epsilon
dt+\frac{1}{\epsilon}F(X^\epsilon, Y^\epsilon,
\theta_t^\epsilon(\theta_s^\epsilon\omega))dt,  \\
&&dY^\epsilon=BY^\epsilon dt+G(X^\epsilon, Y^\epsilon,
\theta_t^\epsilon(\theta_s^\epsilon\omega))dt,
\end{eqnarray*}
with initial datum $Z(0)=(X(0), Y(0))=Z^\epsilon(s, \omega, Z_0).$
Thus, $Z^\epsilon(t+s, \omega, Z_0)=Z^\epsilon(t,
\theta_s^\epsilon\omega, Z^\epsilon(s, \omega, Z_0)).$ Since
$Z^\epsilon(\cdot, \omega, Z_0)\in C^-_{\frac{-\mu}{\epsilon}},
Z^\epsilon(t, \theta_s^\epsilon\omega, Z^\epsilon(s, \omega,
Z_0))\in C^-_{\frac{-\mu}{\epsilon}}.$  Therefore, $Z^\epsilon(s,
\omega, Z_0)\in M^\epsilon(\theta_s^\epsilon\omega).$ This completes
the proof.
%%%%%     \bigskip
%%%%%     Using the Lemma \ref{If and only if} and the same argument as in
%%%%%     Duan et al. \cite{Duan1} one can show
%%%%%     $$M^\epsilon(\omega)=\left\{\big(H^\epsilon(\omega, Y_0), Y_0\big): Y_0\in
%%%%%     H_2\right\}$$ is a random invariant set, i.e.,
%%%%%     $\Phi^\epsilon\left(t, \omega, M^\epsilon(\omega)\right)\subset
%%%%%     M^\epsilon(\theta_t\omega)$ and thus $M^\epsilon(\omega)$ is a
%%%%%     Lipschitz invariant manifold.
\end{proof}

\bigskip
\begin{remark}
We remark that the key point in the proof of Theorem \ref{Th-RIM} is
that
\begin{eqnarray*}
\nonumber\kappa(K, \gamma_1, \gamma_2, \mu,
\epsilon)&=&\frac{K}{\gamma_1-\mu}+\frac{\epsilon
K}{\mu-\epsilon\gamma_2} < 1.
\end{eqnarray*}
In the particular case  where $\epsilon=1,$ one has
$\kappa=\frac{K}{\gamma_1-\mu}+\frac{ K}{\mu-\gamma_2} < 1$, which
is the usual spectral gap condition. We note also that  the proof is
valid for sufficiently small $\epsilon>0$ only in the case
$\frac{K}{\gamma_1-\mu}<1$. This explains the assumption
\textbf{(A3)}.   It is unclear to us about how to relax this
condition.
\end{remark}

\bigskip
In what follows we prove the exponential tracking property which
means the manifold $M^\epsilon(\omega)$ attracts exponentially all
the orbits of $\Phi^\epsilon$ on condition that the scaling
parameter is sufficiently small. %%%%%     Instead (\textbf{A3}) we assume a
%%%%%     strong hypotheses (\textbf{A3s}) holds:
%%%%%     \begin{equation} \label{Strong-Decay-Lip}
%%%%%      2L\leq \gamma_1.
%%%%%     \end{equation}

\bigskip

\begin{theorem}\label{Th-Exponential-Tracking} (Exponential tracking property)\\
Assume that the assumptions (\textbf{A1})--(\textbf{A3}) hold. Then
for sufficiently small $\epsilon>0$, the Lipschitz invariant
manifold for
\eqref{Fast-Equation-Random}--\eqref{Slow-Equation-Random} obtained
in Theorem \ref{Th-RIM} has the exponential {tracking property} in
the following sense: There exist constants $C_1>0$ and $C_2>0$ such
that for any $Z_0=(X_0, Y_0)\in H$ there is a $\bar{Z}_0=(\bar{X_0},
\bar{Y}_0)\in M^\epsilon(\omega)$ such that
\begin{equation*}%%%%%     \label{Exponent-Track}
\|\Phi^\epsilon(t, \omega, Z_0)-\Phi^\epsilon(t, \omega,
\bar{Z}_0)\|\leq C_1e^{-C_2t}\|Z_0-\bar{Z}_0\|, \;t\geq0,
\end{equation*}
where $\|\cdot\|$ denotes the norm in space $\mathbb{H}= H_1\times
H_2$ defined by $$\|z\|=\|x\|_1+\|y\|_2, \;\; z=(x, y).$$
\end{theorem}

\begin{proof}
Assume that $Z^\epsilon(t)=( X^\epsilon(t), Y^\epsilon(t))$ and
$\bar{Z}^\epsilon(t)=( \bar{X}^\epsilon(t), \bar{Y}^\epsilon(t))$
are two solutions for
\eqref{Fast-Equation-Random}--\eqref{Slow-Equation-Random}, then
$\mathcal
{Z}^\epsilon(t)=\bar{Z}^\epsilon(t)-Z^\epsilon(t):=(U^\epsilon(t),
V^\epsilon(t))$ satisfies the equations
\begin{eqnarray}
\label{Tracking-Equation-Fast}
&&dU^\epsilon=\frac{1}{\epsilon}AU^\epsilon
dt+\frac{1}{\epsilon}\tilde{F}(U^\epsilon, V^\epsilon,
\theta_t^\epsilon\omega)dt,\\
&&dV^\epsilon=BV^\epsilon dt+\tilde{G}(U^\epsilon, V^\epsilon,
\theta_t^\epsilon\omega)dt,\label{Tracking-Equation-Slow}
\end{eqnarray}
where
\begin{equation*}
\tilde{F}(U^\epsilon, V^\epsilon,
\theta_t^\epsilon\omega)=F(U^\epsilon+X^\epsilon,
V^\epsilon+Y^\epsilon, \theta_t^\epsilon\omega)-F(X^\epsilon,
Y^\epsilon, \theta_t^\epsilon\omega),
\end{equation*}
and
\begin{equation*}
\tilde{G}(U^\epsilon, V^\epsilon,
\theta_t^\epsilon\omega)=G(U^\epsilon+X^\epsilon,
V^\epsilon+Y^\epsilon, \theta_t^\epsilon\omega)-G(X^\epsilon,
Y^\epsilon, \theta_t^\epsilon\omega).
\end{equation*}
First we claim that $\mathcal {Z}^\epsilon(t)=(U^\epsilon(t),
V^\epsilon(t))$ is a solution of
\eqref{Tracking-Equation-Fast}--\eqref{Tracking-Equation-Slow} in
$C_{-\frac{\mu}{\epsilon}}^{+}$ if %%%%%%     and only if
\begin{equation}\label{Tracking-Equi-Fast-Slow-Random}
\left(
\begin{array}{ccc}
U^\epsilon(t)\\
\\
V^\epsilon(t)
\end{array}
\right)= \left(
\begin{array}{ccc}
e^{At/\epsilon}U^\epsilon(0)+\frac{1}{\epsilon}\int_0^te^{A(t-s)/\epsilon}\tilde{F}(U^\epsilon(s), V^\epsilon(s), \theta_s^\epsilon\omega)ds\\
\\
\int_{+\infty}^te^{B(t-s)}\tilde{G}(U^\epsilon(s), V^\epsilon(s),
\theta_s^\epsilon\omega)ds
\end{array}
\right).
\end{equation}
This can be   verified by using the variation of constants formula.
Next we are going to prove that
\eqref{Tracking-Equi-Fast-Slow-Random} has solutions $(U^\epsilon,
V^\epsilon)$ in $C_{-\frac{\mu}{\epsilon}}^+$ with $(U^\epsilon(0),
V^\epsilon(0))=(U(0), V(0))$ and such that  $(\bar{X}_0,
\bar{Y}_0)=(U(0), V(0))+(X_0, Y_0)\in M^\epsilon(\omega)$. Recall
that
\begin{equation*}
(\bar{X}_0, \bar{Y}_0)\in M^\epsilon(\omega)\Longleftrightarrow
\bar{X}_0=\frac{1}{\epsilon}\int_{-\infty}^0e^{A(-s)}F(X^\epsilon(s,
\bar{Y}_0),Y^\epsilon(s,\bar{Y}_0), \theta_s^\epsilon\omega)ds.
\end{equation*}
It follows that
\begin{equation*}
(\bar{X}_0, \bar{Y}_0)=(U(0), V(0))+(X_0, Y_0)\in M^\epsilon(\omega)
\end{equation*}
if and only if
\begin{eqnarray}\nonumber
U(0)&=&-X_0+\frac{1}{\epsilon}\int_{-\infty}^0e^{A(-s)}F(X^\epsilon(s,
V(0)+Y_0), Y^\epsilon(s,V(0)+Y_0), \theta_s^\epsilon\omega)ds\\
&=&-X_0+H^\epsilon(\omega, V(0)+Y_0).\label{Equiv-Invariant}
\end{eqnarray}
For every $\mathcal {Z}=(U, V)\in C_{-\frac{\mu}{\epsilon}}^+$
define for $t\geq 0$
\begin{equation*}
\mathcal {I}_1^\epsilon(\mathcal
{Z}(\cdot))[t]:=e^{At/\epsilon}U(0)+\frac{1}{\epsilon}\int_0^te^{A(t-s)/\epsilon}
\tilde{F}\left(U(s),V(s), \theta_s^\epsilon\omega\right)ds,
\end{equation*}
and
\begin{equation*}
\mathcal {I}_2^\epsilon(\mathcal
{Z}(\cdot))[t]:=\int_{+\infty}^te^{B(t-s)/\epsilon}
\tilde{G}(U(s),V(s), \theta_s^\epsilon\omega)ds,
\end{equation*}
where $U(0)$ is given by \eqref{Equiv-Invariant}. Define the
operator $\mathcal{I}^\epsilon:C_{-\frac{\mu}{\epsilon}}^+\mapsto
C_{-\frac{\mu}{\epsilon}}^+$ by
\begin{equation*}
\mathcal{I}^\epsilon(\mathcal {Z}(\cdot)):= \left(
\begin{array}{ccc}
\mathcal
{I}_1^\epsilon(\mathcal {Z}(\cdot))\\
\mathcal {I}_2^\epsilon(\mathcal {Z}(\cdot))
\end{array}
\right)
\end{equation*}
Assuming that $\mathcal {Z}, \bar{\mathcal {Z}}\in
C_{-\frac{\mu}{\epsilon}}^+$, we obtain from \eqref{Equiv-Invariant}
the estimate
\begin{eqnarray}\nonumber
\|e^{At/\epsilon}(U(0)-\bar{U}(0))\|_1&\leq& e^{-\gamma_1
t/\epsilon} Lip H^{\epsilon}\|V(0)-\bar{V}(0)\|_2\\
\nonumber&\leq&e^{-\gamma_1 t/\epsilon} Lip
H^{\epsilon}\left\|\int_{+\infty}^0e^{B(-s)}\left(\tilde{G}(\mathcal
{Z}(s), \theta_s^\epsilon\omega)-\tilde{G}(\bar{\mathcal {Z}}(s),
\theta_s^\epsilon\omega)\right)ds\right\|_2\\
\nonumber&\leq&e^{-\gamma_1 t/\epsilon} Lip H^{\epsilon}\cdot
K\int_0^{+\infty} e^{\gamma_2s}\|\mathcal {Z}(s)-\bar{\mathcal
{Z}}(s)\|ds,
\end{eqnarray}
and so
\begin{eqnarray}\nonumber
\|\mathcal{I}_1^\epsilon(\mathcal {Z}-\bar{\mathcal {Z}})\|_{C^{+,
1}_{-\frac{\mu}{\epsilon}}}&\leq& LipH^{\epsilon}\cdot K\|\mathcal
{Z}-\bar{\mathcal
{Z}}\|_{C^+_{-\frac{\mu}{\epsilon}}}\sup\limits_{t\geq
0}\left\{e^{-(-\frac{\mu}{\epsilon}+\frac{\gamma_1}{\epsilon})t}\int_0^{+\infty}
e^{\left(\gamma_2-\frac{\mu}{\epsilon}\right)s}ds\right\}\\
\nonumber&&+\frac{K}{\epsilon}\|\mathcal {Z}-\bar{\mathcal
{Z}}\|_{C^+_{-\frac{\mu}{\epsilon}}}\sup\limits_{t\geq
0}\left\{e^{\frac{\mu}{\epsilon}t}\int_0^te^{-\gamma_1(t-s)/\epsilon}e^{-\frac{\mu}{\epsilon}s}ds\right\}\\
\label{I-1-Lip}&\leq& \left(\frac{Lip H^\epsilon\cdot \epsilon
K}{\mu-\epsilon\gamma_2}+\frac{K}{\gamma_1-\mu}\right)\|\mathcal
{Z}-\bar{\mathcal {Z}}\|_{C^+_{-\frac{\mu}{\epsilon}}}.
\end{eqnarray}
 For the operator
$\mathcal{I}_2^\epsilon$ we have
\begin{eqnarray}\nonumber
\|\mathcal{I}_2^\epsilon(\mathcal {Z}-\bar{\mathcal {Z}})\|_{C^{+,
2}_{-\frac{\mu}{\epsilon}}}&\leq&K \|\mathcal {Z}-\bar{\mathcal
{Z}}\|_{C^+_{-\frac{\mu}{\epsilon}}}\sup\limits_{t\geq
0}\left\{e^{-(-\frac{\mu}{\epsilon}+\gamma_2)t}\int_t^{+\infty}
e^{(-\frac{\mu}{\epsilon}+\gamma_2)s}ds\right\}\\
&\leq&\frac{\epsilon K}{\mu-\gamma_2}\|\mathcal {Z}-\bar{\mathcal
{Z}}\|_{C^+_{-\frac{\mu}{\epsilon}}}\label{I-2-Lip}.
\end{eqnarray}
Recalling that $$Lip
H^\epsilon\leq\frac{K}{\left(\gamma_1-\mu\right)\left[1-K\left(\frac{1}{\gamma_1-\mu}
+\frac{\epsilon}{\mu-\epsilon\gamma_2}\right)\right]}$$ and taking
\eqref{I-1-Lip} and \eqref{I-2-Lip} into account, we obtain
\begin{equation*}
\|\mathcal{I}^\epsilon(\mathcal {Z}-\bar{\mathcal
{Z}})\|_{C^+_{-\frac{\mu}{\epsilon}}}\leq \rho(K, \gamma_1,
\gamma_2, \mu, \epsilon)\|\mathcal {Z}-\bar{\mathcal
{Z}}\|_{C^+_{-\frac{\mu}{\epsilon}}}
\end{equation*}
with
\begin{eqnarray*}
\rho(K, \gamma_1, \gamma_2, \mu, \epsilon)&=& \frac{
K}{\gamma_1-\mu}+\frac{\epsilon K}{\mu-\epsilon\gamma_2}
\\ &&+\frac{K^2}{\left(\gamma_1-\mu\right)\left(\frac{\mu}{\epsilon}-\gamma_2\right)
\left[1-K\left(\frac{1}{\gamma_1-\mu}+\frac{\epsilon}{\mu-\epsilon\gamma_2}\right)\right]}\\
&\rightarrow& \frac{K}{\gamma_1-\mu}
\end{eqnarray*}
as $\epsilon \rightarrow 0$. By \eqref{Condition-mu} there is a
sufficiently small constant $\epsilon'_0>0$ such that $\rho(K,
\gamma_1, \gamma_2, \mu, \epsilon)<1$ for all
$0<\epsilon<\epsilon'_0$. Therefore, the operator $\mathcal
{I}^\epsilon$ is strictly contractive and has a unique fixed point
$\mathcal {Z}\in C_{-\frac{\mu}{\epsilon}}^+$ which is the unique
solution for \eqref{Tracking-Equi-Fast-Slow-Random} and satisfies
$(\bar{X}_0, \bar{Y}_0)=(U(0), V(0))+(X_0, Y_0)\in
M^\epsilon(\omega)$. Moreover, we have
\begin{equation*}
\|\mathcal {Z}\|_{C_{-\frac{\mu}{\epsilon}}^+}\leq
\frac{1}{1-\left(\frac{K}{\gamma_1-\mu}+\frac{\epsilon
K}{\mu-\epsilon\gamma_2}\right)}\|\mathcal {Z}(0)\|
\end{equation*}
which means
\begin{equation*}
\|\Phi^\epsilon(t, \omega, Z_0)-\Phi^\epsilon(t, \omega,
\bar{Z}_0)\|\leq\frac{e^{-\frac{\mu}{\epsilon}
t}}{1-\left(\frac{K}{\gamma_1-\mu}+\frac{\epsilon
K}{\mu-\epsilon\gamma_2}\right)}\|Z_0-\bar{Z}_0\| ,\; t>0.
\end{equation*}
Therefore, the exponential tracking property of $M^\epsilon(\omega)$
is obtained.
\end{proof}

\medskip

\begin{remark}\label{Remark-1}
By the relationship between solutions of system
\eqref{Fast-Equation-Stoch}--\eqref{Slow-Equation-Stoch} and
\eqref{Fast-Equation-Random}--\eqref{Slow-Equation-Random}, the
original fast-slow stochastic system also has a Lipschitz random
invariant manifold under the conditions of Theorem \ref{Th-RIM},
which is represented as
\begin{eqnarray*}
\mathcal{M}^\epsilon(\omega)&=&M^\epsilon(\omega)+(\eta^{\frac{1}{\epsilon}}(\omega),
0)\\
%%%%%     &=&\left\{\left(H^\epsilon(\omega,
%%%%%     Y_0)+\eta^{\frac{1}{\epsilon}}(\theta_t\omega), Y_0\right): Y_0\in
%%%%%     H_2\right\}\\
&=&\big\{\left(h^\epsilon(\omega, Y_0), Y_0\right): Y_0\in H_2\big\}
\end{eqnarray*}
with $$h^\epsilon(\omega, Y_0)=H^\epsilon(\omega,
Y_0)+\eta^{\frac{1}{\epsilon}}(\omega).$$ Hence, if system
\eqref{Fast-Equation-Random}--\eqref{Slow-Equation-Random} has an
exponential tracking manifold so has system
\eqref{Fast-Equation-Stoch}--\eqref{Slow-Equation-Stoch}.
\end{remark}

\medskip

\begin{remark}
Theorem \ref{Th-Exponential-Tracking} implies that   any orbit of
the fast--slow system tends exponentially to an orbit on the
manifold $M^\epsilon(\omega)$ which is governed by an evolutionary
equation with usual time scale. To be more specific, i.e., for any
solution $Z^\epsilon=(X^\epsilon, Y^\epsilon)$ for
\eqref{Fast-Equation-Random}--\eqref{Slow-Equation-Random}, there is
an orbit $\tilde{Z}^\epsilon(t, \omega)=(\tilde{X}^\epsilon(t,
\omega), \tilde{Y}^\epsilon(t, \omega))$ on the manifold
$M^\epsilon$ which satisfies the evolutionary equation
\begin{equation*}
\dot{\tilde{Y}}^\epsilon=B\tilde{Y}^\epsilon+G\left(H^\epsilon(
\theta_t^\epsilon\omega, \tilde{Y}^\epsilon), \tilde{Y}^\epsilon,
\theta_t^\epsilon\omega\right)
\end{equation*}
such that
\begin{equation*}
\|Z^\epsilon(t, \omega)-\tilde{Z}^\epsilon(t, \omega
)\|\leq\frac{e^{-\frac{\mu}{\epsilon}t}}{1-\left(\frac{K}{\gamma_1-\mu}+\frac{\epsilon
K}{\mu-\epsilon\gamma_2}\right)}\|Z_0-\tilde{Z}_0\| ,\; t>0,
\end{equation*}
where $Z_0=(X^\epsilon(0), Y^\epsilon(0))$ and $\tilde{Z}_0=(\tilde{
X}(0), \tilde{Y}(0))$.
\end{remark}
Applying the ideas from the Remark \ref{Remark-1} we have
  a reduction system which describes the long-time behavior for
system \eqref{Fast-Equation-Stoch}--\eqref{Slow-Equation-Stoch}.

\medskip

\begin{theorem}\label{Reduction-Theo} (Reduction system)\\
Assume that $\epsilon>0$ is sufficiently small and the assumption
(\textbf{A1})--(\textbf{A3}) hold. Then for any solution
$z^\epsilon(t)=(x^\epsilon(t), y^\epsilon(t))$ with initial data
$z^\epsilon(0)=(x_0, y_0)$ to system
\eqref{Fast-Equation-Stoch}--\eqref{Slow-Equation-Stoch}, there
exists a solution $\tilde{z}^\epsilon(t)=(\tilde{x}^\epsilon(t),
\tilde{y}^\epsilon(t))$ with $\tilde{z}(0)=(\tilde{x}^\epsilon(0),
\tilde{y}^\epsilon(0))=(\tilde{x}_0, \tilde{y}_0) $ to the reduced
system
\begin{equation*}
\begin{cases}
\dot{\tilde{y}}^\epsilon=B\tilde{y}^\epsilon+g\left(\tilde{x},
\tilde{y}^\epsilon\right),\\
\tilde{x}=h^\epsilon( \theta_t^\epsilon\omega, \tilde{y}^\epsilon),
\end{cases}
\end{equation*}
such that for any $t\geq0$ and almost sure $\omega\in \Omega,$
\begin{eqnarray*}
\|z^\epsilon(t, \omega)-\tilde{z}^\epsilon(t, \omega
)\|&\leq&\frac{e^{-\frac{\mu}{\epsilon}t}}{1-\left(\frac{K}{\gamma_1-\mu}+\frac{\epsilon K}{\mu-\epsilon\gamma_2}\right)}\|z_0-\tilde{z}_0\|\\
&\leq& C_{K, \gamma_1, \mu}e^{\frac{-\mu
t}{\epsilon}}\|z_0-\tilde{z}_0\|
\end{eqnarray*}
with $\frac{-\mu}{\epsilon}<0$  and $C_{K, \gamma_1, \mu}$ being a
constant depending on $K, \gamma_1$ and $\mu$.
\end{theorem}

%%%%%%%%%%%%%%%%%%%%%%%%%%%%%%%%%%%%%%%%%%%%%%%%%%%%%%%%%
\section{Critical manifolds}

Now we consider an asymptotic approximation for the slow manifold
$M^\epsilon(\omega)$, as $\epsilon \to 0$.

The scaling $t\rightarrow \epsilon t$ in system
\eqref{Fast-Equation-Random}-\eqref{Slow-Equation-Random} yields
\begin{eqnarray}\label{Scaling-Fast-Equation-Random}
&&dX^\epsilon=A X^\epsilon dt+F(X^\epsilon, Y^\epsilon,
\theta_{\epsilon t}^\epsilon\omega)dt,\\
\label{Scaling-Slow-Equation-Random}&&dY^\epsilon=\epsilon
BY^\epsilon dt+\epsilon G(X^\epsilon, Y^\epsilon, \theta_{\epsilon
t}^\epsilon\omega)dt,
\end{eqnarray}
where
\begin{eqnarray*}
&&F(X^\epsilon, Y^\epsilon, \theta_{\epsilon
t}^\epsilon\omega)=f(X^\epsilon+\eta^\frac{1}{\epsilon}(\theta_{\epsilon
t}\omega),
Y^\epsilon),\\
&&G(X^\epsilon, Y^\epsilon, \theta_{\epsilon
t}^\epsilon\omega)=g(X^\epsilon+\eta^\frac{1}{\epsilon}(\theta_{\epsilon
t}\omega), Y^\epsilon).
\end{eqnarray*}
We now replace $\eta^\frac{1}{\epsilon}(\theta_{\epsilon t}\omega)$
by $\xi(\theta_t\omega)$ that has the same distribution, then we
have a random evolutionary system, whose solution's distribution
coincides with that of the system
\eqref{Scaling-Fast-Equation-Random}-\eqref{Scaling-Slow-Equation-Random},
in the form of
\begin{eqnarray}
&&d\breve{X}^\epsilon=A \breve{X}^\epsilon
dt+f(\breve{X}^\epsilon+\xi(\theta_t\omega), \breve{Y}^\epsilon)dt,\label{tilde-fast}\\
&&d\breve{Y}^\epsilon=\epsilon B \breve{Y}^\epsilon dt+\epsilon
g(\breve{X}^\epsilon+\xi(\theta_t\omega),
\breve{Y}^\epsilon)dt.\label{tilde-slow}
\end{eqnarray}
By proceeding as in the proof of Theorem \ref{Th-RIM}, it can be
shown that the system \eqref{tilde-fast}-\eqref{tilde-slow} has a
random slow manifold represented as
\begin{equation*}
\breve{M}^\epsilon(\omega)=\left\{\big(\breve{H}^\epsilon(\omega,
Y_0), Y_0\big): Y_0\in H_2\right\}
\end{equation*}
with
\begin{equation*}
\breve{H}^\epsilon(\omega,
Y_0)=\int_{-\infty}^0e^{-As}f(\breve{X}^\epsilon(s, \omega, Y_0
)+\xi(\theta_t\omega), \breve{Y}^\epsilon(s, \omega, Y_0))ds,
\end{equation*}
where
\begin{eqnarray*}
&&\breve{X}^\epsilon(t, \omega,
Y_0)=\int_{-\infty}^te^{A(t-s)}f(\breve{X}^\epsilon(s, \omega,
Y_0)+\xi(\theta_s\omega), \breve{Y}^\epsilon(s, \omega,
Y_0))ds,\;t\leq 0,\\
&&\breve{Y}^\epsilon(t, \omega, Y_0)=e^{Bt\epsilon
}Y_0+\epsilon\int_0^te^{B(t-s)\epsilon}g(\breve{X}^\epsilon(s,
\omega, Y_0)+\xi(\theta_s\omega), \breve{Y}^\epsilon(s, \omega,
Y_0))ds,\;t\leq 0,
\end{eqnarray*}
is the unique solution in $C_{-\mu}^-$  for the above integral
equations. With a change of variables $s/\epsilon\rightarrow t$ in
\eqref{H-epsilon}, we have
\begin{eqnarray*}
H^\epsilon(\omega,
Y_0)&=&\int_{-\infty}^0e^{-As}f\big(X^\epsilon(s\epsilon, \omega,
Y_0)+\eta^{\frac{1}{\epsilon}}(\theta_{\epsilon t}\omega),
Y^\epsilon(s\epsilon, \omega, Y_0)\big)ds\\
&\backsimeq&\breve{H}^\epsilon(\omega, Y_0),
\end{eqnarray*}
where $\backsimeq$ denotes equivalence (coincidence) in
distribution. Therefore, the invariant manifold
$\breve{M}^\epsilon(\omega)$ is a version in distribution for
${M}^\epsilon(\omega)$.

Next, we  show that there exists a random invariant manifold
$M^0(\omega)$, which is called a \emph{random critical manifold} for
system \eqref{tilde-fast}-\eqref{tilde-slow}, will be the asymptotic
limit of the manifold $\breve{M}^\epsilon(\omega)$ as
$\epsilon\rightarrow 0$. To this end, we consider the following
system
\begin{eqnarray}
&&d\bar{X}=A\bar{X}dt+f(\bar{X}+\xi(\theta_t\omega),\bar{Y})dt,\label{epsilon=0-1}\\
&&d\bar{Y}=0.\label{epsilon=0-2}
\end{eqnarray}
By the same discussion in Theorem \ref{Th-RIM}, system
\eqref{epsilon=0-1}-\eqref{epsilon=0-2} has a random invariant
manifold with representation
\begin{equation} \label{M000}
\bar{M}^0(\omega)=\left\{\big(\bar{H}^0(\omega, Y_0), Y_0\big):
Y_0\in H_2\right\},
\end{equation}
where $$\bar{H}^0(\omega,
Y_0)=\int_{-\infty}^0e^{-As}f\big(\bar{X}(s, \omega,
Y_0)+\xi(\theta_s\omega), Y_0\big)ds,$$ and $\bar{X}(t, \omega,
Y_0)$ is the unique solution in $C_{-\mu}^{1, -}$ for integral
equation
\begin{equation*}
\bar{X}(t, \omega, Y_0)=\int_{-\infty}^te^{A(t-s)}f\big(\bar{X}(s,
\omega, Y_0)+\xi(\theta_s\omega), Y_0\big)ds,\;t\leq 0.
\end{equation*}

\medskip

The main result of this section is the following theorem.

\begin{theorem}\label{Approximation thoerem} (Critical manifolds)\\
Let the assumptions (\textbf{A1})--(\textbf{A3}) hold and also
assume that there exists a positive number $C_g$ such that
$\sup\limits_{x\in H_1, y\in H_2}\|g(x, y)\|_{H_2}=C_g$. Then the
invariant manifold $\breve{M}^\epsilon(\omega)$ for the system
\eqref{Scaling-Fast-Equation-Random}-\eqref{Scaling-Slow-Equation-Random}
can be approximated by a critical manifold $\bar{M}^0(\omega)$
defined in \eqref{M000}, in the sense that their respective graph
mappings $\breve{H}^\epsilon  $ and $ \bar{H}^0 $ satisfy
\begin{equation*}
\|\breve{H}^\epsilon(\omega, Y_0)-\bar{H}^0(\omega,
Y_0)\|_1=\mathcal{O}(\epsilon),
\end{equation*}
or
\begin{equation*}
\breve{H}^\epsilon(\omega, Y_0)=\bar{H}^0(\omega, Y_0) +
\mathcal{O}(\epsilon),
\end{equation*}
for all $Y_0\in \mathcal{D}(B)$, a.s. $\;\omega\in \Omega$ and as
$\;\epsilon\rightarrow 0$.
\end{theorem}

\begin{proof}
In this proof, the letter $C$   with or without subscripts   denotes
positive constants whose value may change in different occasions. We
will write the dependence of constant on parameters explicitly if it
is essential. As is known \cite{Pazy}, if $Y_0\in \mathcal {D}(B),
t\leq 0$,
\begin{eqnarray}
\|e^{Bt\epsilon}Y_0-Y_0\|_2&=&\|\int_{\epsilon
t}^0e^{B\tau}BY_0d\tau\|_2\nonumber\\
&\leq&\|BY_0\|_2\int_{\epsilon t}^0e^{-\gamma_2\tau}d\tau\nonumber\\
&=&\|BY_0\|_2\frac{1}{\gamma_2}(e^{-\gamma_2\epsilon
t}-1).\label{Semigroup-Property}
\end{eqnarray}
Then we have, for all $t\leq 0$,
\begin{eqnarray}
\|\breve{Y}^\epsilon(t, \omega, Y_0)-Y_0\|_2&\leq&\|e^{B\epsilon
t}Y_0-Y_0\|_2\nonumber\\
&&+\epsilon\|\int_t^0e^{B\epsilon(t-s)}g\big(\breve{X}^\epsilon(s,
\omega, Y_0)+\xi(\theta_t\omega), \breve{Y}^\epsilon(s, \omega,
Y_0)\big)ds\|_2\nonumber\\
&\leq&\|BY_0\|_2\frac{1}{\gamma_2}(e^{-\gamma_2\epsilon
t}-1)+\epsilon C_g\int_t^0e^{-\epsilon\gamma_2(t-s)}ds\nonumber\\
&=&C(e^{-\gamma_2\epsilon t}-1).\label{Difference-Y-small}
\end{eqnarray}
Then, by using again \eqref{Semigroup-Property}, we have
\begin{eqnarray*}
\|\breve{X}^\epsilon(t, \omega, Y_0)-\bar{X}(t, \omega,
Y_0)\|_1&\leq&K\int_{-\infty}^te^{-\gamma_1(t-s)}\|\breve{X}^\epsilon(s,
\omega, Y_0)-\bar{X}(s, \omega, Y_0)\|_1ds\\
&&+KC\int_{-\infty}^te^{-\gamma_1(t-s)}(e^{-\gamma_2\epsilon
t}-1)ds\\
&=&K\int_{-\infty}^te^{-\gamma_1(t-s)}\|\breve{X}^\epsilon(s,
\omega,
Y_0)-\bar{X}(s, \omega, Y_0)\|_1ds\\
&&+C(\frac{1}{\gamma_1-\epsilon\gamma_2}e^{-\epsilon\gamma_2t}-\frac{1}{\gamma_1}),
\end{eqnarray*}
which implies
\begin{eqnarray}
\|\breve{X}^\epsilon(\cdot, \omega, Y_0)-\bar{X}(\cdot, \omega,
Y_0)\|_{C_{-\mu}^{1, -}}&\leq&K\|\breve{X}^\epsilon(\cdot, \omega,
Y_0)-\bar{X}(\cdot, \omega, Y_0)\|_{C_{-\mu}^{1,
-}}\nonumber\\
&&\cdot\sup\limits_{t\leq0}\int_{-\infty}^te^{-(\gamma_1-\mu)(t-s)}ds\nonumber\\
&&+C\sup\limits_{t\leq 0}\big\{e^{\mu
t}(\frac{1}{\gamma_1-\epsilon\gamma_2}e^{-\epsilon\gamma_2t}-\frac{1}{\gamma_1})\big\}\nonumber\\
&=&\frac{K}{\gamma_1-\mu}\|\breve{X}^\epsilon(\cdot, \omega,
Y_0)-\bar{X}(\cdot, \omega, Y_0)\|_{C_{-\mu}^{1,
-}}+C\sup\limits_{t\leq 0}\mathscr{S}(t,
\epsilon),\nonumber\\
\label{Difference-X}
\end{eqnarray}
where
\begin{equation*}
\mathscr{S}(t, \epsilon)=e^{\mu
t}(\frac{1}{\gamma_1-\epsilon\gamma_2}e^{-\epsilon\gamma_2t}-\frac{1}{\gamma_1}),\;
t\in (-\infty, 0].
\end{equation*}
Since
\begin{eqnarray}
\frac{d\mathscr{S}(t,\epsilon)}{dt}&=& e^{\mu
t}\big(\frac{\mu-\epsilon\gamma_2}{\gamma_1-\epsilon\gamma_2}
e^{-\epsilon\gamma_2 t}-\frac{\mu}{\gamma_1}\big)\nonumber,\;
t\in(-\infty, 0),
\end{eqnarray}
we have, for sufficiently small $\epsilon>0$, that
\begin{eqnarray}
\sup\limits_{t\leq 0}\mathscr{S}(t,
\epsilon)&=&\mathscr{S}\left(-\frac{1}{\epsilon
\gamma_2}\ln\frac{\mu(\gamma_1-\epsilon\gamma_2)}{\gamma_1(\mu-\epsilon\gamma_2)},
\epsilon\right)\nonumber\\
&=&\left(\frac{\mu}{\gamma_1(\mu-\epsilon\gamma_2)}-\frac{1}{\gamma_1}\right)\cdot
\left[\frac{\mu(\gamma_1-\epsilon\gamma_2)}{\gamma_1(\mu-\epsilon\gamma_2)}
\right]^{-\frac{\mu}{\epsilon\gamma_2}}\nonumber\\
&\leq&\frac{\mu}{\gamma_1(\mu-\epsilon\gamma_2)}-\frac{1}{\gamma_1}\label{S-decrease}
\end{eqnarray}
Now, according to \eqref{Difference-X} and \eqref{S-decrease}, we
have
\begin{eqnarray*}
\|\breve{X}^\epsilon(\cdot, \omega, Y_0)-\bar{X}(\cdot, \omega,
Y_0)\|_{C_{-\mu}^{1,
-}}&\leq&\frac{K}{\gamma_1-\mu}\|\breve{X}^\epsilon(\cdot, \omega,
Y_0)-\bar{X}(\cdot, \omega, Y_0)\|_{C_{-\mu}^{1,
-}}\nonumber\\
&&+C(\frac{\mu}{\gamma_1(\mu-\epsilon\gamma_2)}-\frac{1}{\gamma_1}).
\end{eqnarray*}
By \eqref{Condition-mu},
\begin{eqnarray}\label{Difference-X-small}
\|\breve{X}^\epsilon(\cdot, \omega, Y_0)-\bar{X}(\cdot, \omega,
Y_0)\|_{C_{-\mu}^{1, -}}\leq
C(\frac{\mu}{\gamma_1(\mu-\epsilon\gamma_2)}-\frac{1}{\gamma_1}).
\end{eqnarray}
Hence, thanks to \eqref{Difference-Y-small} and
\eqref{Difference-X-small}, {we deduce}
\begin{eqnarray*}
\|\breve{H}^\epsilon(\omega, Y_0)-\bar{H}(\omega, Y_0)\|_1&\leq&
K\int_{-\infty}^0e^{\gamma_1s}\|\breve{X}^\epsilon(s, \omega,
Y_0)-\bar{X}(s, \omega, Y_0)\|_1ds\\
&&+K\int_{-\infty}^0e^{\gamma_1s}\|\breve{Y}^\epsilon(s, \omega,
Y_0)-Y_0\|_2ds\\
&\leq&C(\frac{\mu}{\gamma_1(\mu-\epsilon\gamma_2)}-\frac{1}{\gamma_1})\int_{-\infty}^0e^{(\gamma_1-\mu)s}ds\\
&&+C\int_{-\infty}^0e^{\gamma_1s}(e^{-\gamma_2\epsilon s}-1)ds\\
&=&C(\frac{\mu}{\gamma_1(\mu-\epsilon\gamma_2)}-\frac{1}{\gamma_1})+C(\frac{1}{\gamma_1-\epsilon\gamma_2}-\frac{1}{\gamma_1})\nonumber
\\&=&\mathcal {O}(\epsilon).
\end{eqnarray*}
This completes the proof.
\end{proof}

\medskip

\begin{remark}
%%%%%%    Set $$h(\omega, Y_0)=H(\omega, Y_0)+\xi(\theta_t\omega)$$
Consider the case that $H_2$ is a finite dimensional space, the
operator $B$ is a constant matrix and
\eqref{tilde-fast}--\eqref{tilde-slow} is a coupled system of an
evolutionary equation and ordinary differential equations.
 This system arises from biology, such as Hodgkin-Huxley systems \emph{(}see Example
\ref{Example-2}\emph{)}. Then the above theorem implies  that for
any bounded set $E\subset H_2$,
\begin{equation*}
\sup\limits_{Y_0\in E}\|\breve{H}^\epsilon(\omega,
Y_0)-\bar{H}(\omega, Y_0)\|_1=\mathcal {O}(\epsilon), \; a.s.  \;
\omega\in \Omega\; as\; \epsilon  \rightarrow  0.
\end{equation*}
\end{remark}

\section{Illustrative examples}

Let us look at several examples to illustrate the
  results in the previous two sections.

\begin{example}\label{Example-1}
Let $D\subset \mathbb{R}^3$ be a bounded domain with smooth boundary
$\partial D $.   Consider a coupled  system of stochastic
  parabolic-hyperbolic partial differential equations \emph{(}see, e.g.,
  \emph{\cite{Caraballo}} and \emph{\cite{Chueshov-2}}  \emph{)}
\begin{eqnarray}
&&u_t= \frac{1}{\epsilon}( \Delta u-\alpha u)+\frac{1}{\epsilon}f(u,
v,
v_t)+\frac{1}{\sqrt{\epsilon}}\dot{w}(t), \label{Example-Fast}\\
&&u=0 \;\;on\;\;\partial D,\label{Example-Fast-Boundary}\\
&&v_{tt}=\Delta v-\beta v+g(u, v, v_t),\label{Example-Slow}\\
&&v=0 \;\;on\;\;\partial D,\label{Example-Slow-Boundary}
\end{eqnarray}
where $\Delta$ denotes the Laplace operator and the parameters
$\alpha, \beta$ are positive. The interaction functions
$$f: \mathbb{R}^3\mapsto \mathbb{R} \;\;and \;\; g: \mathbb{R}^3\mapsto
\mathbb{R}$$ are assumed to be Lipschitz continuous with a Lipschitz
constant $K>0$. Thus the assumption (\textbf{A2}) holds. Such a
system may describe a thermoelastic wave propagation in a random
medium \cite{Chow1973}. The wave profile $v$ in an interacting
random thermoelastic medium is
  described by a hyperbolic partial differential equation. If
the  wave   is temperature dependent and the heat conductivity has
faster evolution, then the hyperbolic equation is coupled to a
stochastic parabolic (heat) equation with different characteristic
timescales.

We introduce the usual Hilbert space $L^2(D)$ as well as the Sobolev
spaces ${H}^2(D)$ and ${H}_0^1(D)$.
 Take $H_1=L^2(D)$. Let $A=\Delta -\alpha I_{id}$
with domain $\mathcal {D}(A)={H}^2\cap{H}_0^1.$ By the semigroup
theory the operator $A$ generates a contraction semigroup
$\big\{e^{At}: t\geq 0\big\}$ in $H_1$  \emph{(\cite{Pazy})} which
satisfies $\|e^{At}\|_{H_1} \leq e^{-\gamma_1 t}, t\geq 0$ with
$\gamma_1=\alpha.$ Let $B=\Delta-\beta I_{id}$ with domain $D(B)=
{H}^2\cap{H}_0^1.$ Define
\begin{equation*}
z:=\left(
\begin{array}{cc}
v\\
v'
\end{array}
\right), \mathcal{B}:=\left(
\begin{array}{cc}
0 \;\;\;I_{id}\\
B \;\;\;0
\end{array}
\right)
\end{equation*}
and $H_2=H_0^1(D) \times L^2(D)$ with the energy norm
\begin{equation*}
\|z\|_{H_2}=\big\{\|v\|_{H_0^1}^2+\|v'\|_{L^2}^2\big\}^\frac{1}{2},
\end{equation*}
where $\|\cdot\|_{H_0^1}$ and $\|\cdot\|_{L^2}$ denote the norm in
$H^1_0$ and $L^2$, respectively. Let $\mathcal {D}(\mathcal
{B})=\mathcal {D}(B)\times H^1.$ It is known that $\mathcal {B}$
generates a  unitary group \emph{(\cite{Yosida})} in $H_2$ which
satisfies $\|e^{\mathcal{B}t}\|_{H_2}\leq e^{-\gamma_2 t}$ for $t\in
\mathbb{R}$ with $\gamma_2=0.$ Then the system
\eqref{Example-Fast}--\eqref{Example-Slow-Boundary} can be rewritten
as
\begin{eqnarray*}
&&u_t=\frac{1}{\epsilon}Au+f(u,
z)+\frac{1}{\sqrt{\epsilon}}\dot{w}_t,\\
&&z_t=\mathcal{B}z+G(u, z),
\end{eqnarray*}
with $$G(u, z)=\big(0, g(u, z)\big),$$ which is in the standard form
of \eqref{Fast-Equation-Stoch}--\eqref{Slow-Equation-Stoch}. Thus
under the condition
$$K< \gamma_1$$
and the scaling parameter $\epsilon$ small enough the random
dynamical system generated by
\eqref{Example-Fast}--\eqref{Example-Slow-Boundary} has an invariant
manifold $\mathcal {M}^\epsilon(\omega)=\{(h^\epsilon(\omega, Y_0),
Y_0 )\: Y_0\in  H_2 \}$,   possessing the exponential tracking
property by Theorem \ref{Th-Exponential-Tracking}. Moreover, by
Theorem \ref{Reduction-Theo}, the reduction system for long-time
behavior to system
\eqref{Example-Fast}-\eqref{Example-Slow-Boundary} is
\begin{equation*}
\begin{cases}
\dot{\tilde{y}}^\epsilon=\mathcal
{B}\tilde{y}^\epsilon+G\left(\tilde{x},
\tilde{y}^\epsilon\right),\\
\tilde{x}=h^\epsilon( \theta_t^\epsilon\omega, \tilde{y}^\epsilon).
\end{cases}
\end{equation*}
Note that a similar result for this example has also been obtained
in \emph{\cite{Caraballo}}.
%%%%%     Moreover, if $f$ is
%%%%%     independent of $u$, the fast manifold $M^\epsilon(\omega)$ for
%%%%%     system \eqref{Example-Fast}--\eqref{Example-Slow-Boundary} tend to
%%%%%     an associated  slow manifold $M(\omega)$ as $\epsilon$ tends to
%%%%%     zero.
\end{example}

\bigskip

\begin{example}\label{Example-2}
Let $D\subset \mathbb{R}^n$ be a bounded domain with smooth boundary
$\partial D $. Consider a system of coupled parabolic partial
differential equations and ordinary differential equations
\begin{eqnarray}
&&u_t=\frac{1}{\epsilon}\Delta u+\frac{1}{\epsilon}f(u,
v)+\frac{1}{\sqrt{\epsilon}}\dot{w}_t,\label{Example-2-fast}\\
&&u=0\; on\; \partial D,\\
&&v_t=g(u, v)\label{Example-2-slow},
\end{eqnarray}
where $f:\mathbb{R}^{1+m}\mapsto \mathbb{R}$, $g:
\mathbb{R}^{1+m}\mapsto \mathbb{R}^m$ are Lipschitz maps with a
Lipschitz constant $K>0$:
\begin{eqnarray*}
&&|f(x_1, y_1)-f(x_2, y_2)|\leq
K(|x_1-x_2|+|y_1-y_2|_{\mathbb{R}^m}),\\
&&|g(x_1, y_1)-g(x_2, y_2)|_{\mathbb{R}^m}\leq
K(|x_1-x_2|+|y_1-y_2|_{\mathbb{R}^m}),
\end{eqnarray*}
for all $(x, y)\in \mathbb{R}\times\mathbb{R}^m$.  Thus the
assumption (\textbf{A2}) holds.  This system may model certain
biological processes. For instance, the famous FitzHugh-Nagumo
system \cite{Fitzhugh, Nagumo}, as a simplified version of the
Hodgkin-Huxley model \cite{Cronin},  which describes   mechanisms of
a neural excitability and excitation   for macro-receptors, belongs
to this class.

As in Example \ref{Example-1} the differential operator $A=\Delta$
with domain $\mathcal {D}(A)={H}^2\cap{H}_0^1$ generates a
$C_0$-semigroup $\{e^{At}: t\geq 0\}$ on $H_1=L^2(D)$ which
satisfies $\|e^{At}\|_{H_1}\leq e^{-\gamma_1t}$ with $\gamma_1=\inf
spec\{A\}>0$. Let $B\equiv 0$ in $H_2 =[L^2(D)]^m$. It is clear that
$e^{Bt}=I_{id}$ for all $t\in \mathbb{R}$ and $\|e^{Bt}\|_{H_2}\leq
e^{-\gamma_2t}$ with $\gamma_2=0$. Therefore, the system
\eqref{Example-2-fast}--\eqref{Example-2-slow}   has a random
invariant manifold $\mathcal
{M}^\epsilon(\omega)=\{(h^\epsilon(\omega, Y_0), Y_0 )\: Y_0\in
H_2\}$ with an exponential tracking property if
 $K<\gamma_1$ and $\epsilon>0$ is small enough. We also have the
 reduction system
\begin{equation*}
\begin{cases}
\dot{\tilde{y}}^\epsilon=g\left(\tilde{x},
\tilde{y}^\epsilon\right),\\
\tilde{x}=h^\epsilon( \theta_t^\epsilon\omega, \tilde{y}^\epsilon),
\end{cases}
\end{equation*}
for the long time behavior of the original system
\eqref{Example-2-fast}-\eqref{Example-2-slow}.

\end{example}

\bigskip
\begin{example}
Consider the following system of two coupled wave equations (i.e.,
hyperbolic partial differential equations) on a bounded spatial
interval $I=[0, \pi]:$
\begin{eqnarray}
&&u_{tt}=\frac{1}{\epsilon}(\Delta u-\nu u_t)+\frac{1}{\epsilon}f(u,
v,
v_t)+\frac{1}{\sqrt{\epsilon}}\dot{w}(t),\label{Example-3-fast}\\
&&u=0\;\; on\; \partial I,\label{Example-3-Fast-Boundary}\\
&&v_{tt}=\Delta v-\beta v+g(u, v, v_t),\label{Example-3-slow}\\
&&v=0 \;\;on\;\;\partial I,\label{Example-3-Slow-Boundary}
\end{eqnarray}
where $\Delta$ denotes the Laplace operator and the parameters
$\beta, \nu$ are positive. The interaction functions
$$f: \mathbb{R}^3\mapsto \mathbb{R} \;\; and \;\; g: \mathbb{R}^3\mapsto
\mathbb{R}$$ are Lipschitz continuous with a Lipschitz constant
$K>0$. Thus the assumption (\textbf{A2}) holds. This system models,
for example, vibrating strings connected in parallel with zero
boundary conditions \cite{Na} and multi-component wave fields such
as electromagnetic waves in plasmas, elastic waves in solids, light
waves in anisotropic and inhomogeneous media \cite{Littlejohn}.

Rewrite the equations
\eqref{Example-3-fast}--\eqref{Example-3-Fast-Boundary} as
\begin{eqnarray*}
\frac{d U}{dt}=\frac{1}{\epsilon}\mathcal {A}^\epsilon
U+\frac{1}{\epsilon}F(U, V)+\frac{1}{\sqrt{\epsilon}}\dot{W}(t),
\end{eqnarray*}
where \begin{eqnarray*} &&\mathcal
{A}^\epsilon=\left(\begin{array}{cc}
0 \;\;\;\epsilon I_{id}\\
\Delta \;\;\;-\nu
\end{array}
\right), F(U, V)=\left(\begin{array}{cc}
0\\
f(u, v, v')
\end{array}\right), \; \dot{W}(t)=\left(\begin{array}{cc}0\\
\dot{w}{(t)}\end{array}\right),
\end{eqnarray*}
and
\begin{equation*}
U=(u, u'), V=(v, v')\in H_0^1(0, \pi)\times L^2(0, \pi).
\end{equation*}
The linear operator $\mathcal {A}^\epsilon$ has the eigenvalues
\begin{equation*}
\lambda_k^{\pm}=\frac{\nu\pm\sqrt{\nu^2-4k^2\epsilon}}{2}, \; k=1,
2,...
\end{equation*}
with the corresponding eigenvectors
\begin{equation*}
e_k^\pm =\left(\begin{array}{cc} \sin kx\\
\lambda_k^\pm\sin kx\end{array}\right).
\end{equation*}
It is clear   that the operator $\mathcal {A}^\epsilon$ generates a
$C_0-$semigroup $e^{\mathcal {A}^\epsilon t}$ on Hilbert space
$H_1:=H_0^1(0, \pi)\times L^2(0, \pi)$ equipped with energy norm
introduced in Example \ref{Example-1}, and it satisfies
\begin{equation*}
\|e^{\mathcal{A}^\epsilon t}\|_{H_1}\leq e^{-\nu t}, \; t\geq 0.
\end{equation*}
In the same way as in Example \ref{Example-1} the linear part of the
equation \eqref{Example-3-slow}--\eqref{Example-3-Slow-Boundary}
generates a unitary $C_0-$semigroup on Hilbert space $H_2=H_0^1(0,
\pi)\times L^2(0, \pi)$. Thus under the condition that $K< \nu$, the
system \eqref{Example-3-fast}--\eqref{Example-3-Slow-Boundary} has
an exponentially tracking random invariant manifold $\mathcal
{M}^\epsilon(\omega)=\{(h^\epsilon(\omega, Y_0), Y_0 )\: Y_0\in H_2
\}$ when $\epsilon>0$ is sufficiently small. In particular, by
Theorem \ref{Reduction-Theo} the system
\eqref{Example-3-fast}-\eqref{Example-3-Slow-Boundary} has a
reduction equation
\begin{equation*}
\begin{cases}
\dot{\tilde{y}}^\epsilon=\mathcal
{B}\tilde{y}^\epsilon+G\left(\tilde{x},
\tilde{y}^\epsilon\right),\\
\tilde{x}=h^\epsilon( \theta_t^\epsilon\omega, \tilde{y}^\epsilon),
\end{cases}
\end{equation*}
where $\mathcal {B}$ and $G$ are defined as in Example
\ref{Example-1}.
\end{example}

\medskip
\section{Remarks on the case of local Lipschitz nonlinearity}

We have limited ourselves to the case where the
 nonlinearities are globally Lipschitz continuous.  We remark that when
the nonlinearities in those three examples in Section 6 are only
locally Lipschitz (say, near the origin $(0, 0)$), the above
discussions remain valid locally. To this end we state the
definition of a local random invariant manifold  \cite{Blomker,
Chen}.

\medskip

\begin{definition}
We say that the random dynamical system $\phi(t, \omega)$ has a
local random invariant manifold (LRIM) with radius $R$, if there is
a random set $\mathcal {M}^R(\omega)$, which is defined by the graph
of a  random continuous function $\psi(\omega, \cdot):
\overline{B_R(0)}\bigcap H_2\rightarrow H_1$, such that for all
bounded sets $B$ in $B_R(0)\subset H_2$ we have
$$\phi(t,\omega)[\mathcal {M}^R(\omega)\bigcap B]\subset\mathcal {M}^R(\theta_t\omega)$$
for all $t\in(0, \tau_0(\omega))$ with
$$\tau_0(\omega)=\tau_0(\omega, B)=\inf\{t\geq 0: \phi(t,\omega)[\mathcal {M}^R(\omega)\bigcap
B]\subset\!\!\!\!\!\!\!/B_R(0)\}.$$
\end{definition}
Let $\chi: H_1\times H_2\rightarrow \mathbb{R}$ be a bounded smooth
function such that
$$ \chi(v_1, v_2)=
\begin{cases}1,\quad \ \ &
  \mbox{if } \|v_1\|_1+\|v_2\|_2 \leq 1, \\
  0, \quad \ \ & \mbox{if }\|v_1\|_1+\|v_2\|_2\geq 2.
\end{cases}
$$
For any positive parameter $R$, we define $\chi_R(v_1,
v_2)=\chi(\frac{v_1}{R}, \frac{v_2}{R})$ for all $(v_1, v_2)\in
H_1\times H_2.$ Let $f^{(R)}(x, y):=\chi_R(x, y)f(x, y), g^{(R)}(x,
y):=\chi_R(x, y)g(x, y).$ For every $R>0$, there must exist a
positive  $K_R$ such that
\begin{equation*}%\label{Fast-Lip}
\|f^{(R)}(x_1, y_1)-f^{(R)}(x_2, y_2)\|_1\leq
K_R(\|x_1-x_2\|_1+\|y_1-y_2\|_2),
\end{equation*}
and
\begin{equation*}%\label{Slow-Lip}
\|g^{(R)}(x_1, y_1)-g^{(R)}(x_2, y_2)\|_2\leq
K_R(\|x_1-x_2\|_1+\|y_1-y_2\|_2).
\end{equation*}
Then the cut-off system of
\eqref{Fast-Equation-Random}-\eqref{Slow-Equation-Random} is as
follows:
\begin{eqnarray}\label{Fast-Equation-Random-1}
&&dX^\epsilon=\frac{1}{\epsilon}A X^\epsilon
dt+\frac{1}{\epsilon}F^{(R)}(X^\epsilon, Y^\epsilon,
\theta_t^\epsilon\omega)dt,\\
\label{Slow-Equation-Random-1}&&dY^\epsilon=BY^\epsilon
dt+G^{(R)}(X^\epsilon, Y^\epsilon, \theta_t^\epsilon\omega)dt,
\end{eqnarray}where
\begin{eqnarray*}
&&F^{(R)}(X^\epsilon, Y^\epsilon,
\theta_t^\epsilon\omega)=f^{(R)}(X^\epsilon+\eta^\frac{1}{\epsilon}(\theta_t\omega),
Y^\epsilon),\\
&&G^{(R)}(X^\epsilon, Y^\epsilon,
\theta_t^\epsilon\omega)=g^{(R)}(X^\epsilon+\eta^\frac{1}{\epsilon}(\theta_t\omega),
Y^\epsilon).
\end{eqnarray*}
 The
system
\eqref{Fast-Equation-Random-1}--\eqref{Slow-Equation-Random-1} has a
unique solution   and thus the solution mapping generates a
continuous random dynamical system $\Phi^\epsilon_R$. If
$K_R<\gamma_1$, then the cut-off system
\eqref{Fast-Equation-Random-1}--\eqref{Slow-Equation-Random-1}
admits a globally invariant manifold $\mathcal {M}^{\epsilon}_R$
possessing exponentially tracking property. Now as $\Phi^\epsilon$
and $\Phi^\epsilon_R$ agree on $B_R(0)$, we conclude that
$\widetilde{\mathcal {M}}_R^\epsilon =\mathcal {M}_R^\epsilon
\bigcap B_R(0)$ defines a local invariant manifold of the original
system \eqref{Fast-Equation-Random}-\eqref{Slow-Equation-Random}.

%%%%%     We refer for instance to \cite{Blomker} and \cite {Chen} for
%%%%%     existence of local random invariant manifolds for SPDEs with local
%%%%%     Lipschtiz nonlinearities.

\medskip

{\bf Acknowledgement.} We would like to thank Bjorn Schmalfuss for
pointing out this problem to us and thank Wei Wang, Guanggan Chen
and Jicheng Liu for helpful discussions and comments.

\end{document}